\documentclass[10pt]{article}

\usepackage{amscd,amsmath, amssymb, fancyhdr, mathbbol}

\numberwithin{equation}{section}

\newcommand{\version}{version 2.0,\ \   September 15, 2012}

\def\eqref#1{(\ref{#1})}

\newcommand{\arrow}{{\:\longrightarrow\:}}
\newcommand{\Z}{{\Bbb Z}}
\def\C{{\Bbb C}}

\newcommand{\Q}{{\Bbb Q}}

\def\1{\sqrt{-1}\:}
\newcommand{\restrict}[1]{{\left|_{{\phantom{|}\!\!}_{#1}}\right.}}
\newcommand{\cntrct}                
{\hspace{2pt}\raisebox{1pt}{\text{$\lrcorner$}}\hspace{2pt}}

\newcommand{\calo}{{\cal O}}

\renewcommand{\bar}{\overline}
\renewcommand{\phi}{\varphi}
\renewcommand{\epsilon}{\varepsilon}
\renewcommand{\geq}{\geqslant}
\renewcommand{\leq}{\leqslant}

\newcommand{\Teich}{\operatorname{Teich}}

\newcommand{\Hilb}{\operatorname{Hilb}}

\newcommand{\Pic}{\operatorname{Pic}}

\newcommand{\Diff}{\operatorname{Diff}}

\newcommand{\rk}{\operatorname{rk}}

\newcommand{\Comp}{\operatorname{Comp}}
\newcommand{\Mod}{\operatorname{Mod}}
\newcommand{\Per}{\operatorname{\sf Per}}
\newcommand{\Perspace}{\operatorname{{\mathbb P}\sf er}}

\newcounter{Mycounter}[section]
\newcounter{lemma}[section]
\newcounter{claim}[section]
\renewcommand{\theclaim}{{Claim \thesection.\arabic{claim}}}
\newcommand{\claim}{%
    \setcounter{claim}{\value{Mycounter}}
    \refstepcounter{claim}
    \stepcounter{Mycounter}
    {\noindent \bf \theclaim:\ }}

\newcounter{sublemma}[section]
\newcounter{corollary}[section]
\renewcommand{\thecorollary}{{Corollary \thesection.\arabic{corollary}}}
\newcommand{\corollary}{%
    \setcounter{corollary}{\value{Mycounter}}
    \refstepcounter{corollary}
    \stepcounter{Mycounter}
    {\noindent \bf \thecorollary:\ }}

\newcounter{theorem}[section]
\renewcommand{\thetheorem}{{Theorem \thesection.\arabic{theorem}}}
\newcommand{\theorem}{%
    \setcounter{theorem}{\value{Mycounter}}
    \refstepcounter{theorem}
    \stepcounter{Mycounter}
    {\noindent \bf \thetheorem:\ }}

\newcounter{conjecture}[section]
\renewcommand{\theconjecture}{{Conjecture \thesection.\arabic{conjecture}}}
\newcommand{\conjecture}{%
    \setcounter{conjecture}{\value{Mycounter}}
    \refstepcounter{conjecture}
    \stepcounter{Mycounter}
    {\noindent \bf \theconjecture:\ }}

\newcounter{proposition}[section]
\renewcommand{\theproposition}
      {{Proposition \thesection.\arabic{proposition}}}
\newcommand{\proposition}{%
    \setcounter{proposition}{\value{Mycounter}}
    \refstepcounter{proposition}
    \stepcounter{Mycounter}
    {\noindent \bf \theproposition:\ }}

\newcounter{definition}[section]
\renewcommand{\thedefinition}
      {{Definition~\thesection.\arabic{definition}}}
\newcommand{\definition}{%
    \setcounter{definition}{\value{Mycounter}}
    \refstepcounter{definition}
    \stepcounter{Mycounter}
    {\noindent \bf \thedefinition:\ }}

\newcounter{example}[section]
\newcounter{remark}[section]
\renewcommand{\theremark}{{Remark \thesection.\arabic{remark}}}
\newcommand{\remark}{%
    \setcounter{remark}{\value{Mycounter}}
    \refstepcounter{remark}
    \stepcounter{Mycounter}
    {\noindent \bf \theremark:\ }}

\newcounter{problem}[section]
\newcounter{question}[section]
\makeatletter

\@addtoreset{equation}{section} \@addtoreset{footnote}{section}
\makeatother

\def\blacksquare{\hbox{\vrule width 5pt height 5pt depth 0pt}}
\def\endproof{\blacksquare}

\def\Gr{\mathop{\text{\rm Gr}}}
\def\O{\mathop{\text{\rm O}}}

\begin{document}
\begin{center}
{\LARGE\bf
Families of Lagrangian fibrations on hyperk\"ahler manifolds\\[4mm]
}

Ljudmila Kamenova, Misha
Verbitsky\footnote{Partially supported by 
by RFBR grants 12-01-00944-Á,  10-01-93113-NCNIL-a, and
AG Laboratory NRI-HSE, RF government grant, ag. 11.G34.31.0023.}

\end{center}

{\small \hspace{0.1\linewidth}
\begin{minipage}[t]{0.8\linewidth}
{\bf Abstract} \\
A holomorphic Lagrangian fibration on a 
holomorphically symplectic manifold is a  
holomorphic map with Lagrangian fibers.
It is known (due to Huybrechts) that
a given compact manifold admits only
finitely many holomorphic symplectic
structures, up to deformation. We prove that 
a given compact manifold with $b_2 \geq 7$ admits only 
finitely many deformation types of 
holomorphic Lagrangian fibrations.
We also prove that all known hyperk\"ahler
manifolds are never Kobayashi hyperbolic.
\end{minipage}
}

{\scriptsize
\tableofcontents
}


\section{Introduction}

Irreducible compact hyperk\"ahler manifolds, or irreducible 
holomorphic symplectic manifolds, are a natural generalization of 
K3 surfaces in higher dimensions. The geometry of K3 surfaces is well 
studied. In particular, it is known that any two K3 surfaces are 
deformation equivalent to each other, i.e., there is only one 
deformation type of K3 surfaces. 

A natural question to ask is whether the same is true in higher dimensions. 
The answer is negative due to Beauville's examples. In every possible 
complex dimension $2n$ there are at least the Hilbert scheme of $n$ points on 
a K3 surface $S$, $Hilb^n(S)$, and the generalized Kummer varieties 
$K^{n+1}(A)$, where $A$ is an Abelian surface. These two examples are not 
deformation equivalent since they have different Betti numbers. There are 
two more exceptional examples due to K. O'Grady in dimensions $6$ and $10$. 

It is conjectured that in every fixed dimension there are finitely many 
deformation types of irreducible compact hyperk\"ahler manifolds. 
It is also conjectured that every hyperk\"ahler manifold can be deformed 
to one that admits a holomorphic Lagrangian fibration. It would
be interesting to classify the deformation types of tha pairs
$(M,L)$ of a hyperk\"ahler manifold together with a Lagrangian
fibration on it. In the present paper, we show that the
number of deformational classes of such pairs is finite,
if one fixes the smooth manifold undelying $M$.

In 
\cite{_Huybrechts:finiteness_} Huybrechts proved that for a fixed compact 
manifold there are at most finitely many deformation types of 
hyperk\"ahler structures on it. Therefore,
to prove that the number of deformation classes of pairs
$(M,L)$ is finite, it would suffice to prove it when a 
deformational class of $M$ is fixed.

Let $M\stackrel \pi \arrow X$ be a Lagrangian fibration.
Then $X$ is known to be projective, with $H^2(X)=\C$,
hence $\rk \Pic(X)=1$. Therefore, the primitive 
ample bundle $L_X$ on $X$ is unique (up to torsion).
Denote by $L_M$ the semiample bundle $\pi^*(L_X)$ on
$M$. Clearly, $c_1(L_M)^{\rk M}=0$; a (1,1)-class
satisfying this equation is called {\bf parabolic}.
The Lagrangian fibration $M\stackrel \pi \arrow X$ 
is uniquely determined by a class $[c_1(L_M)]\in \Pic(M)$
which is parabolic and semiample (this is due to D. Matsushita,  
\cite{_Matsushita:fibred_};
see \cite{_Sawon_}  for a detailed exposition of an early work on
Lagrangian fibrations).
Therefore, to classify the Lagrangian fibrations
it would suffice to classify pairs $(M,L_M)$,
where $L_M$ is a parabolic semiample line bundle.

We prove that in the Teichm\"uller space of hyperk\"ahler manifolds with 
a fixed parabolic class the pairs admitting
a Lagrangian fibration form a dense and open subset. The other 
main result is that the action of the monodromy group has finitely many 
orbits. As a corollary of these results we obtain that for a fixed compact 
manifold, there are only finitely many deformation types of hyperk\"ahler
structures equipped with a Lagrangian fibration.


\subsection{Hyperk\"ahler manifolds}

\definition  A {\bf hyperk\"ahler manifold}
is a compact, K\"ahler, holomorphically symplectic manifold.

\hfill

\definition
 A hyperk\"ahler manifold $M$ is called
{\bf simple} if $H^1(M)=0$, $H^{2,0}(M)=\C$.

\hfill

\theorem
(Bogomolov's Decomposition Theorem,
\cite{_Bogomolov:decompo_}, \cite{_Besse:Einst_Manifo_}). 
Any hyperk\"ahler manifold admits a finite covering,
which is a product of a torus and several 
simple hyperk\"ahler manifolds. 
\endproof

\hfill

\remark
Further on, all hyperk\"ahler manifolds
are assumed to be simple.

\hfill

{\bf A note on terminology.}
Speaking of hyperk\"ahler manifolds, people
usually mean one of two different notions.
One either speaks of holomorphically
symplectic K\"ahler manifold, or of 
a manifold with a {\em hyperk\"ahler structure},
that is, a triple of complex structures
satisfying quaternionic relations and
parallel with respect to the Levi-Civita
connection. The equivalence
(in compact case) between these two 
notions is provided by the Yau's solution
of Calabi's conjecture 
(\cite{_Besse:Einst_Manifo_}). Throughout this paper,
we use the complex algebraic geometry
point of view, where ``hyperk\"ahler''
is synonymous with ``K\"ahler 
holomorphically symplectic'', 
in lieu of the differential-geometric
approach. The reader may
check \cite{_Besse:Einst_Manifo_} for an introduction
to hyperk\"ahler geometry from the
differential-geometric point of view.

Notice also that we included compactness in our definition
of a hyperk\"ahler manifold. In the differential-geometric 
setting, one does not usually assume that the
manifold is compact.

\subsection{The Bogomolov-Beauville-Fujiki form}

\theorem
(\cite{_Fujiki:HK_})
Let $\eta\in H^2(M)$, and $\dim M=2n$, where $M$ is
hyperk\"ahler. Then $\int_M \eta^{2n}= c q(\eta,\eta)^n$,
for some integer quadratic form $q$ on $H^2(M)$ and a constant $c>0$.
\endproof

\hfill

\definition
This form is called
{\bf  Bogomolov-Beauville-Fujiki form}.  It is defined
by this relation uniquely, up to a sign. The sign is determined
{}from the following formula (Bogomolov, Beauville;
\cite{_Beauville_}, \cite{_Huybrechts:lec_}, 23.5)
\begin{align*}   \lambda q(\eta,\eta) &=
   (n/2)\int_X \eta\wedge\eta  \wedge \Omega^{n-1}
   \wedge \bar \Omega^{n-1} +\\
 &+(1-n)\frac{\left(\int_X \eta \wedge \Omega^{n-1}\wedge \bar
   \Omega^{n}\right) \left(\int_X \eta \wedge
   \Omega^{n}\wedge \bar \Omega^{n-1}\right)}{\int_M \Omega^{n}
   \wedge \bar \Omega^{n}}
\end{align*}
where $\Omega$ is the holomorphic symplectic form,
and $\lambda$ a positive constant.

\hfill

\remark
The form $q$ has signature $(3,b_2-3)$.
It is negative definite on primitive forms, and positive
definite on the space $\langle \Omega, \bar \Omega, \omega\rangle$
 where $\omega$ is a K\"ahler form, as seen from the
following formula
\begin{multline}\label{_BBF_via_Kahler_Equation_}
   \mu q(\eta_1,\eta_2)= \\
   \int_X \omega^{2n-2}\wedge \eta_1\wedge\eta_2  
   - \frac{2n-2}{(2n-1)^2}
   \frac{\int_X \omega^{2n-1}\wedge\eta_1 \cdot
   \int_X\omega^{2n-1}\wedge\eta_2}{\int_M\omega^{2n}}, \
   \  \mu>0
\end{multline}
(see e. g. \cite{_Verbitsky:cohomo_}, Theorem 6.1,
or \cite{_Huybrechts:lec_}, Corollary 23.9).

\hfill

\definition
Let $[\eta]\in H^{1,1}(M)$ be a real (1,1)-class on
a hyperk\"ahler manifold $M$. We say that $[\eta]$
is {\bf parabolic} if $q([\eta],[\eta])=0$.
A line bundle $L$ is called {\bf parabolic} if $c_1(L)$
is parabolic.

\subsection{The hyperk\"ahler SYZ conjecture}

\theorem
(D. Matsushita, see \cite{_Matsushita:fibred_}).
Let $\pi:\; M \arrow X$ be a surjective holomorphic map
{}from a hyperk\"ahler manifold $M$ to $X$, with $0<\dim X < \dim M$.
Then $\dim X = 1/2 \dim M$, and the fibers of $\pi$ are 
holomorphic Lagrangian tori (this means that the symplectic
form vanishes on the fibers).\footnote{Here, as elsewhere,
we assume that the hyperk\"ahler manifold $M$ is simple.}

\hfill

\definition Such a map is called
{\bf a holomorphic Lagrangian fibration}.

\hfill

\remark The base of $\pi$ is conjectured to be
rational. J.-M. Hwang (\cite{_Hwang:base_}) 
proved that $X\cong \C P^n$, if it is smooth.
D. Matsushita (\cite{_Matsushita:CP^n_}) 
proved that it has the same rational cohomology
as $\C P^n$.

\hfill

\remark
 The base of $\pi$ has a natural flat connection
on the smooth locus of $\pi$. The combinatorics of this connection
can be used to determine the topology of $M$ 
(\cite{_Kontsevich-Soibelman:torus_},  
\cite{_Gross:SYZ_}),

\hfill

\definition
 Let $(M,\omega)$ be a Calabi-Yau manifold,
$\Omega$ the holomorphic volume form, and $Z\subset M$ a real 
analytic subvariety, Lagrangian with respect to
$\omega$. If $\Omega\restrict Z$ is proportional to
the Riemannian volume form, $Z$ is called {\bf special
Lagrangian} (SpLag).

\hfill

The special Lagrangian varieties were defined
in  \cite{_Harvey_Lawson:Calibrated_}
by Harvey and Lawson, who proved that they
minimize the Riemannian
volume in their cohomology class. This implies, in
particular, that their moduli are finite-dimensional.
In \cite{_McLean:SpLag_}, McLean studied deformations
of non-singular special Lagrangian
subvarieties and showed that they are unobstructed.

In \cite{_SYZ:MS_is_T_du_}, Strominger-Yau-Zaslow 
tried to explain the mirror symmetry phenomenon
using the special Lagrangian fibrations. They
conjectured that any Calabi-Yau manifold admits
a Lagrangian fibration with special Lagrangian fibers.
Taking its dual fibration, one obtains ``the mirror dual''
Calabi-Yau manifold.

\hfill

\definition A line bundle is called
{\bf semiample} if  $L^N$ is generated
by its holomorphic sections, which have 
no common zeros.

\hfill

\remark From semiampleness 
it obviously follows that $L$ is nef. Indeed,
let $\pi:\; M \arrow {\Bbb P}H^0(L^N)^*$ be the standard
map. Since the sections of $L$ have no common zeros, $\pi$ is 
holomorphic. Then $L\cong \pi^* \calo(1)$, and the
curvature of $L$ is a pullback of 
the K\"ahler form on $\C P^n$. However,
the converse is false: a nef bundle is not 
necessarily semiample 
(see e.g. \cite[Example 1.7]{_Demailly_Peternell_Schneider:nef_}).

\hfill

\remark Let $\pi:\; M \arrow X$ 
be a holomorphic Lagrangian fibration, and $\omega_X$
a K\"ahler class on $X$. Then $\eta:=\pi^*\omega_X$ is 
semiample and parabolic. The converse is also
true, by Matsushita's theorem:
if $L$ is semiample and parabolic, $L$ induces a Lagrangian
fibration. This is the only known 
source of non-trivial special Lagrangian fibrations.

\hfill

\conjecture\label{_SYZ_conj_Conjecture_}
(Hyperk\"ahler SYZ conjecture)
Let $L$ be a parabolic nef line bundle
on a hyperk\"ahler manifold. Then
$L$ is semiample.

\hfill

\remark
This conjecture was stated by many people
(Tyurin, Bogomolov, Hassett-Tschinkel,
Huybrechts, Sawon); please see 
\cite{_Sawon_} for an interesting
and historically important discussion,
and \cite{_Verbitsky:SYZ_} 
for details and references.

\hfill

\remark
The SYZ conjecture can be seen as
a hyperk\"ahler version of the ``abundance conjecture''
(see e.g. \cite{_Demailly_Peternell_Schneider:ps-eff_}, 
2.7.2). 

\hfill

\claim\label{_examples_Claim_}
Let $M$ be an irreducible hyperk\"ahler manifold
in one of 4 known classes, that is, a deformation
of a Hilbert scheme of points on K3, a deformation
of generalized Kummer variety, or a deformation
of one of two examples by O'Grady. Then 
$M$ admits a deformation equipped with 
a holomorphic Lagrangian fibration.

\hfill

{\bf Proof:}
When $S$ is an elliptic K3 surface, the Hilbert scheme of points
$\Hilb^n(S)$ has an induced Lagrangian fibration with smooth fibers that
are products of $n$ elliptic curves:
$\Hilb^n(S) \rightarrow Sym^n({\Bbb P}^1) \simeq {\Bbb P}^n$.
Similarly, when $A$ is an elliptic Abelian surface, the generalized
Kummer variety $K^n(A)$ admits a Lagrangian fibration. 
Another construction gives
 Lagrangian fibrations on $\Hilb^n(S)$ and on $K^n(A)$ if
$S$ contains a smooth genus $n$ curve and if $A$ contains a smooth genus
$n+2$ curve (see examples 3.6 and 3.8 in \cite{_Sawon_}).

 O'Grady's examples are 
deformation equivalent to Lagrangian fibrations,
as follows from Corollary 1.1.10 in 
\cite{_Rapagnetta_}. 
\endproof

\section{Hyperk\"ahler geometry: preliminary results}


\subsection{Teichm\"uller space and the moduli space}

Here we cite the relevant result from the deformation
theory of hyperk\"ahler manifolds.
We follow \cite{_V:Torelli_}.

Let $M$ be a hyperk\"ahler manifold (compact and simple, as usual), and 
$\Comp_0$ be the Fr\`echet manifold of all complex structures of
hyperk\"ahler type on~$M$. The quotient $\Teich:=\Comp_0/\Diff^0$ of
$\Comp_0$ by isotopies is a finite-dimensional complex analytic space
(\cite{_Catanese:moduli_}). This 
quotient is called the {\bf Teichm\"uller space} of $M$. When $M$ is a
complex curve, the quotient $\Comp_0/\Diff^0$ is the Teichm\"uller space
of this curve.

The mapping class group $\Gamma=\Diff^+/\Diff^0$ acts on $\Teich$ in the 
usual way, and its quotient $\Mod$ is the {\bf moduli space} of $M$.

As shown in \cite{_Huybrechts:finiteness_}, $\Teich$ has a finite
number of connected components. Take a connected component $\Teich^I$
containing a given complex structure~$I$, and let
$\Gamma^I\subset\Gamma$ be the set of elements of $\Gamma$ fixing this
component. Since $\Teich$ has only a finite number of connected
components, $\Gamma^I$ has finite index in $\Gamma$. On the other hand,
as shown in \cite{_V:Torelli_}, the image of the group $\Gamma$ is
commensurable to ${O}\big(H^2(M,\mathbb Z),q\big)$.

In \cite[Lemma 2.6]{_V:Torelli_} it was proved that any hyperk\"ahler
structure on a given simple hyperk\"ahler manifold is also simple.
Therefore, $H^{2,0}(M,I')=\mathbb C$ for all $I'\in\Comp$. This
trivial observation is a key to the following well-known definition.

\hfill

\definition
Let $(M,I)$ be a hyperk\"ahler manifold, and $\Teich$ its Teichm\"uller
space. Consider a map $\Per:\Teich\arrow\mathbb PH^2(M,\mathbb C)$,
sending $J$ to the line $H^{2,0}(M,J)\in\mathbb PH^2(M,\mathbb C)$. It
is easy to see that $\Per$ maps $\Teich$ into the open subset of a
quadric, defined by
\begin{equation*}\label{_perspace_lines_Equation_}
\Perspace:=\big\{l\in\mathbb PH^2(M,\mathbb C)\ \big|\ q(l,l)=0,\
q(l,\bar l)>0\big\}.
\end{equation*}
The map $\Per:\Teich\arrow\Perspace$ is called the {\bf period map},
and the set $\Perspace$ the {\bf period space}.

\hfill

The following fundamental theorem is due to F. Bogomolov
\cite{_Bogomolov:defo_}.

\hfill

\theorem
Let\/ $M$ be a simple hyperk\"ahler manifold, and $\Teich$ its
Teichm\"uller space. Then the period map\/ $\Per:\Teich\arrow\Perspace$
is a local diffeomorphism\/ {\rm(}that is, an etale map\/{\rm).}
Moreover, it is holomorphic.
\endproof

\hfill

\remark
 Bogomolov's theorem implies that $\Teich$ is smooth. However, it
is not necessarily Hausdorff\/ {\rm(}and it is non-Hausdorff even in
the simplest examples\/{\rm).}
\endproof

\subsection{The polarized Teichm\"uller space}

In \cite[Corollary 2.6]{_Verbitsky:parabolic_}, the following
proposition was deduced from \cite{_Boucksom_}
and~\cite{_Demailly_Paun_}.

\hfill

\theorem\label{_Kahler_cone_Pic=1_Theorem_}
{\sl Let\/ $M$ be a simple hyperk\"ahler manifold, such that all
integer\/ $(1,1)$-classes satisfy\/ $q(\nu,\nu)\geq 0$. Then its
K\"ahler cone is one of the two connected components of the set\/
$K:=\big\{\nu\in H^{1,1}(M,\mathbb R)\ \big|\ q(\nu,\nu)>0\big\}$.}
\endproof

\hfill

\remark\label{_Pic=1_Parabolic_Remark_}
{}From \ref{_Kahler_cone_Pic=1_Theorem_}
it follows that on a hyperk\"ahler manifold
with $\Pic(M)=\Z$, for any rational 
class $\eta\in H^{1,1}(M)$ with $q(\eta,\eta)\geq 0$, 
either $\eta$ or $-\eta$ is nef.

\hfill

\remark\label{_closed_divisor_in_Teich_}
Consider an integer vector $\eta\in H^2(M)$ which is positive, that is,
satisfies $q(\eta, \eta)>0$. Denote by $\Teich^\eta$ the set of all
$I\in \Teich$ such that $\eta$ is of type $(1,1)$ on $(M, I)$. The
space $\Teich^\eta$ is a closed divisor in $\Teich$. Indeed, by
Bogomolov's theorem, the period map $\Per:\Teich\arrow\Perspace$ is
etale, but~the image of $\Teich^\eta$ is the set of all $l\in\Perspace$
which are orthogonal to $\eta$; this condition defines a closed divisor
$C_\eta$ in $\Perspace$, hence $\Teich^\eta=\Per^{-1}(C_\eta)$ is also
a closed divisor.

\hfill

\remark\label{_Kaehler_class_eta_}
When $I\in\Teich^\eta$ is generic, Bogomolov's theorem implies that the
space of rational $(1,1)$-classes $H^{1,1}(M,\mathbb Q)$ is
one-dimensional and generated by $\eta$. This is seen from the
following argument.  Locally around a given point $I$ the period map
$\Teich^\eta\arrow\Perspace$ is surjective on the set $\Perspace^\eta$
of all $I\in\Perspace$ for which $\eta\in H^{1,1}(M,I)$. However, the
Hodge-Riemann relations give 
\begin{equation}\label{_Perspace^eta_ortho_Equation_}
\Perspace^\eta=\big\{l\in\Perspace\ \big|\ q(\eta,l)=0\big\}.
\end{equation}
Denote the set of such points of $\Teich^\eta$ by
$\Teich^\eta_{\text{gen}}$. It follows from
\ref{_Kahler_cone_Pic=1_Theorem_} that, for any
$I\in\Teich^\eta_{\text{gen}}$, either $\eta$ or $-\eta$ is a K\"ahler
class on $(M,I)$.

\hfill


Consider a connected component $\Teich^{\eta,I}$ of $\Teich^\eta$.
Changing the sign of $\eta$ if necessary, we may assume that $\eta$ is
K\"ahler on $(M,I)$. By Kodaira's theorem about stability of K\"ahler
classes, $\eta$ is K\"ahler in some neighbourhood
$U\subset\Teich^{\eta,I}$ of $I$. Therefore, the sets
$$V_+:=\big\{I\in\Teich^\eta_{\text{gen}}\ \big|\ \eta\text{ is
K\"ahler on }(M,I)\big\}$$
and
$$V_-:=\big\{I\in\Teich^\eta_{\text{gen}}\ \big|\ -\eta\text{ is
K\"ahler on }(M,I)\big\}$$ 
are open in $\Teich^\eta_{\text{gen}}$. It is easy to see that
$\Teich^\eta_{\text{gen}}$ is a complement to a union of countably many
divisors in $\Teich^\eta$ corresponding to the points
$I'\in\Teich^\eta$ with $\rk\Pic(M,I')>1$. Therefore, for any connected
open subset $U\subset\Teich^\eta$, the intersection
$U\cap\Teich^\eta_{\text{gen}}$ is connected. Since
$\Teich^\eta_{\text{gen}}$ is represented as a disjoint union of open
sets $V_+\sqcup V_-$, every connected component of
$\Teich^\eta_{\text{gen}}$ and of $\Teich^\eta$ is contained in $V_+$
or in $V_-$. We obtained the following corollary.

\hfill

\corollary
{\sl Let\/ $\eta\in H^2(M)$ be a positive integer vector, $\Teich^\eta$
the corresponding divisor in the Teichm\"uller space, and\/
$\Teich^{\eta,I}$ a connected component of\/ $\Teich^\eta$ containing a
complex structure\/ $I$. Assume that\/ $\eta$ is K\"ahler on\/ $(M,I)$.
Then\/ $\eta$ is K\"ahler for all\/ $I'\in\Teich^{\eta,I}$ which
satisfy\/ $\rk H^{1,1}(M,\mathbb Q)=1$.}
\endproof

\hfill

We call the set $\Teich^\eta_{\text{pol}}$ of all $I\in\Teich^\eta$ for
which $\eta$ is K\"ahler the {\bf polarized Teichm\"uller space}, and
$\eta$ its {\bf polarization}. From the above arguments it is clear 
that the polarized Teichm\"uller space $\Teich^\eta_{\text{pol}}$ is 
open and dense in $\Teich^\eta$.

The quotient ${\cal M}_\eta$ of $\Teich^\eta_{\text{pol}}$ by the
subgroup of the mapping class group fixing $\eta$ is called the {\bf
moduli of polarized hyperk\"ahler manifolds}. It is known (due to the
general theory which goes back to Viehweg and Grothen\-dieck) that
${\cal M}_\eta$ is Hausdorff and quasiprojective (see
e.g.~\cite{_Viehweg:moduli_} and \cite{_GHK:moduli_HK_}).

\hfill

\remark\label{_countably_quasiproj_div_Remark_}
We conclude that there are countably many quasiprojective divisors
${\cal M}_\eta$ immersed in the moduli space $\Mod$ of hyperk\"ahler
manifolds. Moreover, every algebraic complex structure belongs to one
of these divisors. However, these divisors need not to be closed.
Indeed, as proven in \cite{_Ananin_Verbitsky_}, each of 
${\cal M}_\eta$ is dense in $\Mod$.

\hfill

In \cite[Theorem 1.7]{_Ananin_Verbitsky_}, the following theorem was proven.

\hfill

\theorem\label{_dense_main_Theorem_}
{\sl Let\/ $M$ be a compact, simple hyperk\"ahler manifold, $\Teich^I$
a connected component of its Teichm\"uller space, and\/
$\Teich^I\stackrel\Psi\arrow\Teich^I/\Gamma^I=\Mod$ its projection to
the moduli space of complex structures. Consider a positive or negative 
vector\/ $\eta\in H^2(M,\mathbb Z)$, and let\/ $\Teich^{I,\eta}$ be the
corresponding connected component of the polarized Teichm\"uller space.
Assume that\/ $b_2(M)>3$. Then the image\/ $\Psi(\Teich^{I,\eta})$ is
dense in\/ $\Mod$.}

\hfill

The proof relies on a more general proposition about lattices. 

\hfill

\proposition \cite[Proposition 3.2, Remark 3.12]{_Ananin_Verbitsky_} 
\label{_lattices_Prop_} 
{\sl Let\/ $V$ be an\/ $\mathbb R$-vector space equipped with a 
non-degene\-rate symmetric form of signature\/ $(s_+,s_-)$ with\/ 
$s_+\ge3$ and\/ $s_-\ge1$. Consider a lattice\/ $L\subset V$. Let\/ 
$\Gamma$ be a subgroup of finite index in\/ $\O(L)$, and\/ $l\in L$. 
Then\/ $\Gamma\cdot\Gr_{++}(l^\perp)$ is dense in $\Gr_{++}(V)$.}

\hfill

\remark
Since the proof of this statement is symmetric in $s_+$ and $s_-$, 
the same proposition is valid if we assume that $s_+\ge1$ and\/ $s_-\ge3$.


\section{Main results}


\subsection{The moduli of manifolds with Lagrangian fibrations}

Here we assume that $b_2(M) \geq 7$ as we need it for our proof of 
\ref{_dense_ope_Lagra_Theorem_}. The authors conjecture that the 
result is valid for smaller Betti numbers as well. 

\hfill

\definition
Let $L$ be a holomorphic line bundle
on a hyperk\"ahler manifold. We call
$L$ {\bf Lagrangian} if it is parabolic and semiample.

\hfill

\definition
Let $M$ be a hyperk\"ahler manifold.
Fix a parabolic class $L\in H^2(M,\Z)$.
We denote by $\Teich_L$ the Teichm\"uller space
of all complex structures $I$ of hyperk\"ahler type 
on $M$ such that $L$ is of type $(1,1)$ on $(M,I)$.
Clearly, $\Teich_L$ is a divisor in the 
whole Teichm\"uller space of $M$.
The space $\Teich_L$ is called
{\bf the Teichm\"uller space of 
hyperk\"ahler manifolds with parabolic class}.

\hfill

Matsushita proves the following openness result in 
\cite[Theorem 1.1]{_Matsushita:openness_}: 

\hfill

\theorem
Let $\Teich_L^\circ\subset\Teich_L$
be the set of all $I\in \Teich_L$ for which
$L$ is Lagrangian. Then $\Teich_L^\circ$
is open in $\Teich_L$. \endproof

\hfill

The main results of the present paper are the following 
two theorems.

\hfill

\theorem\label{_dense_ope_Lagra_Theorem_}
The subspace $\Teich_L^\circ\subset\Teich_L$
is dense and open in $\Teich_L$.

\hfill

\begin{proof}
Fix a positive class $\eta \in H^2(M,\Z)$ and define $\Teich_{L, \eta}^\circ$ 
to be the open subset of  $\Teich_L^\circ$ for which $\eta$ is a polarization. 
Consider the projection $\Psi$ to the moduli space $\Mod$ as defined in 
\ref{_dense_main_Theorem_}. 
Since  ${\cal M}_\eta$ is quasiprojective (see \cite{_Viehweg:moduli_}), 
then $\Psi(\Teich_{L, \eta}^\circ)$ is Zariski open, and therefore dense in 
$\Psi(\Teich_{L, \eta})$. 

Fix a negative vector $L' \in H^2(M,\Z)$ such that the sublattice $<L, L'>$ 
is of rank 2. Notice that 
$\Psi(\Teich_L)= \{ l \in {\Bbb P}H^2(M,\Z) | q(l,l)=0, q(l, \bar l) >0, 
q(L,l)=0 \}/\Gamma_L$ and 
$\Psi(\Teich_{L, \eta})=  \{ l \in {\Bbb P}H^2(M,\Z) | q(l,l)=0, q(l, \bar l) >0, 
q(L,l)=0, q(\eta,l)=0 \} / \Gamma_{L, \eta}$. 
Applying \ref{_lattices_Prop_} to the 
quotient $H^2(M,\Z)/<L, L'>$, we see that $\Psi(\Teich_{L, L', \eta})$ is 
dense in $\Psi(\Teich_{L, L'})$ for any $L'$. Here we needed to assume 
$b_2 \geq 7$, because $H^2(M,\Z)$ is of signature $(3, b_2-3)$ and the 
quotient $H^2(M,\Z)/<L, L'>$ is of signature $(2, b_2-4)$. This satisfies 
the conditions of \ref{_lattices_Prop_} since $b_2 -4 \geq 3$. 

However, 
$\bigcup_{L'} \Psi(\Teich_{L, L'})$ is dense in $\Psi(\Teich_L)$, and 
$\bigcup_{L'} \Psi(\Teich_{L, L', \eta})$ is dense in $\Psi(\Teich_{L, \eta})$. 
Therefore, $\Psi(\Teich_{L, \eta})$ is dense in 
$\Psi(\Teich_L)$ and $\Teich_L^\circ$ is dense in  $\Teich_L$. 
\end{proof}

\hfill

\remark\label{_dense_Lagra_Remark_}
Together with \ref{_lattices_Prop_}, 
\ref{_dense_ope_Lagra_Theorem_} implies that
the set of manifolds with Lagrangian fibrations is dense
within the deformation space of a hyperk\"ahler manifold $M$,
if $M$ admits a Lagrangian fibration.

\hfill

\theorem
Consider the action of the monodromy group
$\Gamma_I$ on $H^2(M,\Z)$, and let $S\subset H^2(M,\Z)$
be the set of all classes which are parabolic
and primitive. Then there are only
finitely many orbits of $\Gamma_I$ on $S$.

\hfill

\begin{proof} 
In the proof we use Nikulin's technique of discriminant-forms 
described in \cite{Nikulin:forms_}. 

Denote by $\Lambda$ the lattice $(H^2(M,\Z), q)$. It is a free $\Z$-module 
of finite rank together with a non-degenerate symmetric bilinear form $q$ 
with values in $\Z$. If $\{ e_i \}_{i \in I}$ is a basis of the lattice 
$\Lambda$, its {\it discriminant} is defined to be 
$\text{discr} (\Lambda) = \text{det} (e_i \cdot e_j)$.  
There is a canonical embedding  $\Lambda \hookrightarrow 
\Lambda^\ast = \text{Hom} (\Lambda, \mathbb Z)$ using the bilinear 
form of $\Lambda$. The {\it discriminant group} 
$A_\Lambda = \Lambda^\ast / \Lambda$ is a finite Abelian group of 
order $|\text{discr}(\Lambda)|$. One can extend the bilinear 
form to $\Lambda^\ast$ with values in $\Q$ and define the 
{\it discriminant-bilinear form} of the lattice $b_\Lambda : A_\Lambda 
\times A_\Lambda \rightarrow \Q / \Z$. It is a finite non-degenerate form. 
A subgroup $H \subset A_\Lambda$ is {\it isotropic} if $q_\Lambda|_H = 0$, 
where $q_\Lambda$ is the quadratic form corresponding to $b_\Lambda$. 
Given any subset $K \subset \Lambda$, its {\it orthogonal complement} is 
$K^\bot = \{v \in \Lambda | (v,K)=0  \}$. 


An embedding of lattices $\Lambda_1 \hookrightarrow \Lambda_2$ is 
{\it primitive} if $\Lambda_2 / \Lambda_1$ is a free $\Z$-module.
Take a primitive vector $v \in \Lambda$ with $q(v)=0$. We can choose a 
vector $f \in \Lambda$ with minimal positive quadratic intersection 
$\alpha = q(v,f)$. Then $0 < \alpha \leq |\text{discr}(\Lambda)|$. 
It is implied by the following lemma: 

\hfill

\begin{lemma}
The minimal positive intersection $\alpha$ divides $\text{discr}(\Lambda)$. 
\end{lemma}

\begin{proof}
Since $v$ is primitive, we can choose a free $\Z$-basis $\{ v_1=v, v_2, \dots, 
v_n \}$ of $\Lambda$, where $n = \text{rk}(\Lambda)$. If 
$\alpha= \text{min} \{q(v,f)|f \in \Z^n \}$, then $\alpha \Z$ is an ideal 
generated by $\{ q(v,v_i), i=1, \dots, n \}$. For every $i =1, \dots, n, ~~
q(v,v_i)=\alpha \cdot a_i$ for some $a_i \in \Z$. Thus the matrix 
$[q(v_j,v_i)]$ has first column divisible by $\alpha$. Then 
$\text{det}[q(v_j,v_i)]=\text{discr}(\Lambda)$ is divisible by $\alpha$. 
\end{proof}

\hfill

Let $K$ be the primitive sublattice of $\Lambda$ spanned by $v$ and $f$. 
The intersection matrix of $\text{Span}(v,f)$ has determinant 
$q(v,v)q(f,f)-q(v,f)^2=-\alpha^2$ which is bounded: 
$-|\text{discr}(\Lambda)|^2 \leq -\alpha^2 < 0$. Since $\rk(K)=2$, 
$K$ has at most four primitive isotropic vectors ($2\rk(K)=4$). 

An {\it overlattice} of $\Lambda$ is a lattice embedding 
$i: \Lambda \rightarrow \Lambda'$ with $\Lambda$ and $\Lambda'$ of the 
same rank, or equivalently, such that $H_{\Lambda'} = \Lambda' / \Lambda$ 
is a finite Abelian group. Note that we have the inclusions: 
$\Lambda \hookrightarrow \Lambda' \hookrightarrow \Lambda'^\ast 
\hookrightarrow \Lambda^\ast$. Therefore, $H_{\Lambda'} \subset \Lambda'^\ast 
/ \Lambda \subset \Lambda^\ast / \Lambda = A_\Lambda$. 

\hfill

\begin{proposition} \cite[Proposition 1.4.1]{Nikulin:forms_} 
\label{_Nikulin_1.4.1_}
The correspondence $\Lambda' \rightarrow H_{\Lambda'}$ determines a 
bijection between overlattices of $\Lambda$ and isotropic subgroups 
of $A_\Lambda$. Furthermore, 
$H_{\Lambda'}^\bot = \Lambda'^\ast / \Lambda$ 
and $H_{\Lambda'}^\bot / H_{\Lambda'} = A_{\Lambda'}$. 
\end{proposition}

\hfill

Let $L=K^\bot$ be the orthogonal complement of $K$ in $\Lambda$. Then 
$K \oplus L \subset \Lambda \subset K^\ast \oplus L^\ast$. Since 
$\det(L)$ is bounded, in view of \ref{_Nikulin_1.4.1_}, 
there are finitely many ways of expressing $\Lambda$ as an overlattice of 
$\Lambda_K \doteq K \oplus K^\bot$ because $A_\Lambda$ is finite of order 
$|\text{discr}(\Lambda)|$ and there are finitely many isotropic subgroups. 

\hfill

Define the lattices $\Lambda$ and $\Lambda'$ to be {\it stably 
equivalent} if there exists a lattice $M$ such that 
$\Lambda \oplus M \simeq \Lambda' \oplus M$. The following proposition 
is a reformulation of Theorem 1.1 in Chapter 9 of Cassels' book 
\cite{_Cassels_}. 

\hfill

\begin{proposition} \label{_finit_stable_equiv_}
There exist only a finite number of lattices stably equivalent to $\Lambda$. 
\end{proposition} \endproof



\hfill

If we assume that there are infinitely many orbits of $\Gamma_I$, this would 
imply that there exist infinitely many non-isomorphic pairs of lattices 
$(K, K^\bot)$. Then for infinitely many of them $K^\bot$ would be stably 
equivalent to $K_1^\bot$ for another $K_1$ since there are only finitely 
many choices for $K$. This contradicts \ref{_finit_stable_equiv_} and 
the result follows. 
\end{proof}

\hfill

\corollary\label{_finitely_many_divi_Corollary_}
For any hyperk\"ahler manifold, there are only
finitely many orbits of $\Gamma_I$ on the set
of all divisors $\Teich_L$ with a parabolic class.

\hfill

Combining \ref{_finitely_many_divi_Corollary_}
and \ref{_dense_ope_Lagra_Theorem_}, we obtain the
following result.

\hfill

\corollary
Let $M$ be a hyperk\"ahler manifold. Then there are
only finitely many deformation types of 
Lagrangian fibrations $(M,I)\arrow S$, for all
complex structures on $M$.

\hfill

\begin{proof}
By \ref{_Kaehler_class_eta_} we can assume that 
$H^{1,1}(M,\mathbb Q)$ is one-dimensional and generated by a parabolic 
class $L$. Since either $L$ or $-L$ is nef, we can assume $L$ to be nef. 
{}From \ref{_dense_ope_Lagra_Theorem_} it follows that for each pair 
$(M,L)$ there exists a unique deformation type of a fibration structure. 
We conclude finiteness of the deformation types of Lagrangian fibrations 
since there are finitely many orbits of $\Gamma_I$ on the set $\Teich_L$. 
\end{proof}

\subsection{Kobayashi hyperbolicity in hyperk\"ahler geometry}

\definition
A compact manifold $M$ is called {\bf Kobayashi hyperbolic}
if any holomorphic map $\C \arrow M$ is constant.

For an introduction to the hyperbolic geometry,
please see \cite{_Lang:hyperbolic_}.

As an application of \ref{_dense_ope_Lagra_Theorem_},
we obtain the following result.

\hfill

\theorem
Let $M$ be an irreducible holomorphic symplectic manifold
in one of 4 known classes known, that is, a deformation
of a Hilbert scheme of points on K3, a deformation
of generalized Kummer variety, or a deformation
of one of two examples by O'Grady. Then $M$ is not
Kobayashi hyperbolic.

\hfill

{\bf Proof:} From Brody's lemma \cite{_Lang:hyperbolic_}
it follows that a limit of non-hyperbolic manifolds
is again non-hyperbolic. Therefore, it would suffice
to find a dense set of non-hyperbolic manifolds within
the moduli space. A hyperk\"ahler manifold
admitting a holomorphic Lagrangian fibration
is non-hyperbolic, because it contains
abelian subvarieties. As follows from \ref{_examples_Claim_}, 
all known types of hyperk\"ahler manifolds
admit a deformation which has a Lagrangian 
fibration. By \ref{_dense_Lagra_Remark_},
such deformations are dense in the moduli.
\endproof

\hfill

It is conjectured that all hyperk\"ahler and
Calabi-Yau manifolds are not hyperbolic.
The strongest result about non-hyperbolicity of
hyperk\"ahler manifolds so far was due to F. Campana,
who proved in \cite{_Campana:twistor_} that 
any twistor family of a hyperk\"ahler manifold
has at least one fiber which is non-hyperbolic.

\hfill

\hfill

\noindent{\bf Acknowledgments.} 
We are grateful to V. Nikulin for many conversations 
on bilinear forms, Valery Gritsenko for a helpful
email, and to F. Bogomolov for his remarks. 
The first named author thanks the Laboratory of Algebraic 
Geometry and its Applications for their kind invitation 
and hospitality during her stay in Moscow in June 2012. 
We are thankful to the Simons Center for Geometry and Physics 
and to John Morgan for inviting the second named author and 
for making our work more enjoyable. Many thanks to
Fr\'ed\'eric Campana, Simone Diverio, Klaus Hulek and 
Daniel Huybrechts for
a discussion of non-hyperbolicity of hyperk\"ahler 
manifolds.

\hfill

{\small

\hfill

\noindent {\sc Ljudmila Kamenova\\
Department of Mathematics, 3-115 \\
Stony Brook University \\
Stony Brook, NY 11794-3651, USA \\
 }

\noindent {\sc Misha Verbitsky\\
Laboratory of Algebraic Geometry, \\
Faculty of Mathematics, National Research University HSE,\\
7 Vavilova Str. Moscow, Russia
 }
}

\end{document}